\documentclass{svmult}

\renewcommand{\email}[1]{\emailname: #1} 

\usepackage{mathptmx}       
\usepackage{helvet}         
\usepackage{courier}        

\usepackage{makeidx}         
\usepackage{graphicx}        
\usepackage[bottom]{footmisc}

\usepackage{latexsym}
\usepackage{amsmath}
\usepackage{amsfonts}
\usepackage{amssymb}
\usepackage{bm}

\usepackage{url}
\usepackage{algorithm}
\usepackage{algorithmic}
\usepackage[misc,geometry]{ifsym}

\renewenvironment{proof}{\noindent{\itshape Proof.}}{\smartqed\qed}

\spdefaulttheorem{assumption}{Assumption}{\upshape \bfseries}{\itshape}
\spdefaulttheorem{algo}{Algorithm}{\upshape \bfseries}{\itshape}

\DeclareSymbolFont{bbold}{U}{bbold}{m}{n}
\DeclareSymbolFontAlphabet{\mathbbold}{bbold}

\providecommand{\esssup}{\operatorname*{ess\ sup}}

\begin{document}

\title*{Approximate Quadrature Measures on Data--Defined Spaces}

\author{H. N. Mhaskar}

\institute{
H. N. Mhaskar (\Letter)
 \at 
Institute of Mathematical Sciences, Claremont Graduate University, Claremont, CA 91711. \\
 \email{hmhaska@gmail.com}
 }

\maketitle

\index{Mhaskar, Hrushikesh N.}

\paragraph{Dedicated to Ian H.~Sloan on the occasion of his 80th birthday.}

\abstract{
An important question in the theory of approximate integration is to
study the conditions on the nodes $x_{k,n}$ and weights $w_{k,n}$ that allow an estimate of the form
$$
\sup_{f\in \mathcal{B}_\gamma}\left|\sum_k w_{k,n}f(x_{k,n})-\int_\mathbb{X} fd\mu^*\right| \le cn^{-\gamma}, \qquad n=1,2,\cdots,
$$
where $\mathbb{X}$ is often a manifold with its volume measure $\mu^*$, and $\mathcal{B}_\gamma$ is the unit ball of a suitably defined smoothness class, parametrized by $\gamma$. 
In this paper, we study this question in the context of a quasi-metric, locally compact, measure space $\mathbb{X}$ with a probability measure $\mu^*$. 
We show that  quadrature formulas exact for integrating the so called diffusion polynomials of degree $<n$ satisfy such estimates. Without requiring exactness, such formulas can be obtained as a solutions of some kernel-based optimization problem. We discuss the connection with the question of optimal covering radius. Our results generalize in some sense many recent results in this direction. 
}

\section{Introduction}

The theory of approximate integration of a function based on finitely many samples of the function is a very old subject. Usually, one requires the necessary quadrature formula to be exact for some finite dimensional space. For example, we mention the following theorem, called Tchakaloff's theorem \cite[Exercise~2.5.8, p.~100]{rivlin1974chebyshev}. (For simplicity of exposition, the notation used in the introduction may not be the same as in the rest of this paper.)
\begin{theorem}\label{tchakalofftheo}
Let $\mathbb{X}$ be a compact topological space, $\{\phi_j\}_{j=0}^{N-1}$ be continuous real valued functions on $\mathbb{X}$, and $\mu^*$ be a probability measure on $\mathbb{X}$ (i.e., $\mu^*$ is a positive Borel measure with $\mu^*(\mathbb{X})=1$). Then there exist $N+1$ points $x_1,\cdots,x_{N+1}$, and non--negative numbers $w_1,\cdots,w_{N+1}$ such that
\begin{equation}\label{tchakaloffquad}
\sum_{k=1}^{N+1} w_k=1, \qquad \sum_{k=1}^{N+1}w_k\phi_j(x_k)=\int_\mathbb{X} \phi_j(x)d\mu^*(x), \qquad j=0,\cdots,N-1.
\end{equation}
\end{theorem}

It is very easy to see that under the conditions of Theorem~\ref{tchakalofftheo}, if $f:\mathbb{X}\to\mathbb{R}$ is continuous, and $V_N=\mathsf{span}\{1,\phi_0,\cdots,\phi_{N-1}\}$ then
\begin{equation}\label{tchakaloffquaderr}
\left|\int_\mathbb{X} f(x)d\mu^*(x)-\sum_{k=1}^{N+1}w_kf(x_k)\right|\le 2\min_{P\in V_N}\max_{x\in\mathbb{X}}|f(x)-P(x)|.
\end{equation}
A great deal of modern research is concerned with various variations of this theme, interesting from the point of view of computation. 
For example, can we ensure all the weights $w_k$ to be equal by a judicious choice of the points $x_k$, 
or can we obtain estimates similar to \eqref{tchakaloffquaderr} with essentially arbitrary points $x_k$, 
with or without requiring that the weights be positive, 
or can we obtain better rates of convergence for subspaces (e.g., suitably defined Bessel potential spaces) of the space of continuous functions than that guaranteed by \eqref{tchakaloffquaderr}?
 Of course, this research typically requires $\mathbb{X}$ to have some additional structure. 

The current paper is motivated by the spate of research within the last couple of years, in particular, by the results in \cite{brauchart2014qmc, 
brauchart2015covering, brandolini2014quadrature, ehler_covering_radius2016, ehler_qmc_grassman2016}. 
The focus in \cite{brauchart2014qmc, brauchart2015covering} is the case when $\mathbb{X}$ is the unit sphere $\mathbb{S}^q$ embedded in $\mathbb{R}^{q+1}$, and the weights are all equal. 
A celebrated result by Bondarenko et. al.  in \cite{bondarenko2010optimal} shows that for every large enough $n$, there exist $\mathcal{O}(n^q)$ points on $\mathbb{S}^q$ such that 
equal weight quadrature formulas based at these points are exact for integrating spherical polynomials of degree $<n$. In these papers (see also \cite{besovquadpap}), it is shown that for such formulas an estimate of the following form holds:
\begin{equation}\label{besovintbd}
\left|\int_{\mathbb{S}^q} f(x)d\mu^*(x)-\frac{1}{n}\sum_k f(x_k)\right|\le
\frac{c}{n^s}\|f\|_{s,p,q},
\end{equation}
where $\|\cdot\|_{s,p,q}$ is a suitably defined Bessel potential subspace of $L^p(\mu^*)$ and the smoothness index $s$ satisfies $s>q/p$, so that functions in this subspace are actually continuous. More generally, the systems of points $x_k$ for which an estimate of the form \eqref{besovintbd} holds is called an approximate QMC design. Various constructions and properties of such designs are studied.
The papers \cite{brandolini2014quadrature, ehler_covering_radius2016} study certain analogous questions in the context of a smooth, compact, Riemannian manifold. The case of a Grassmanian manifold is studied in \cite{ehler_qmc_grassman2016}, with numerical illustrations.

Our paper is also motivated by machine learning considerations, where one is given a data set of the form $\{(x_j,y_j)\}$, sampled from some unknown probability distribution $\mu$. The objective is to approximate $f(x)=\mathbb{E}(y|x)$. In this context, $\mathcal{P}=\{x_j\}$ is usually referred to as a \emph{point cloud}, and is considered to be sampled from the marginal distribution $\mu^*$.  Thus, 
in contrast to the research described above, the points $\{x_j\}$ in this context are \emph{scattered}; i.e., one does not have a choice of
stipulating their locations in advance. 

Typically, the points $\{x_j\}$ are in a high dimensional \emph{ambient space}, and the classical approximation results are inadequate for practical applications. 
 A very powerful relatively recent idea to work with such problems is the notion of a data-defined manifold. Thus, we assume that the points $\{x_j\}$ lie on a low dimensional manifold embedded in the ambient space. 
 This manifold itself is not known, but some relation on the set $\mathcal{P}$ is assumed to be known, giving rise to a graph structure with vertices on the manifold. 
 Various quantities such as the eigenvalues and eigenfunctions of the Laplace-Beltrami (or a more general elliptic partial differential) operator on the manifold can be approximated well by the corresponding objects for the so called graph Laplacian that can be computed directly from the points $\{x_j\}$ themselves (e.g., \cite{singer, lafon, belkin2007convergence, niyogi2, belkinfound, rosasco2010learning}). 
 It is shown in \cite{jones2008parameter} that a local coordinate chart on such data-defined manifolds can be obtained in terms of the heat kernel on the manifold. 
 
  In our theoretical investigations, we will not consider the statistical problem of machine learning, but assume that the marginal distribution $\mu^*$ is known. 
 Since the heat kernel can be approximated well 
 using the eigen-decomposition of the graph Laplacian \cite{coifmanlafondiffusion}, we find it convenient and essential to formulate all our assumptions in this theory only in terms of the  measure $\mu^*$ on the manifold and the heat kernel. In particular, we do not consider the question of estimating the eigenvalues and eigenfunctions of this kernel, but assume that they are given.
 
In \cite{frankbern, modlpmz}, we have studied the existence of quadrature formulas exact for certain eigenspaces in this context. 
They play a critical role in approximation theory based on these eigenspaces; e.g.,  \cite{mauropap, eignet, compbio}.
Although our proofs of the existence of quadrature formulas based on scattered data so far require the notion of gradient on a manifold, the approximation theory itself has been developed in a more general context of  locally compact, quasi-metric, measure spaces.

In this paper, we will prove certain results analogous to those in \cite{brauchart2014qmc, brauchart2015covering} in the context of locally compact, quasi-metric, measure spaces. 
In order to do so, we need to generalize the notion of an approximate QMC design to include non-equal weights, satisfying certain regularity conditions. 
We will show that an estimate of the form \eqref{besovintbd} holds if and only if it holds for what we call diffusion polynomials of certain degree. 
Conversely, if this estimate holds with non-negative weights, then the assumption of regularity is automatically satisfied. 
We will also point out the connection between such quadratures and the so called covering radius of the points on which they are based. 
Our results include their counterparts in \cite{brauchart2014qmc, brauchart2015covering}, except that we deal with slightly larger smoothness classes. 
We will discuss a construction of the approximate quadratures that yield a bound of the form \eqref{besovintbd}, without referring to the eigen-decompositon itself. 

We describe our general set up in Section~\ref{setupsect}, and discuss the main results in Section~\ref{mainsect}. The proofs are given in Section~\ref{pfsect}. Section~\ref{prepsect} reviews some preparatory
results required in the proofs. 

\section{The Set-up}\label{setupsect}

In this section, we describe our general set up. In Sub-section~\ref{spacesect}, we introduce the notion of a data-defined space. In Sub-section~\ref{meassect}, we review some measure theoretic concepts. The smoothness classes in which we study the errors in approximate integration are defined in Sub-section~\ref{smoothsect}.

\subsection{The Quasi-metric Measure Space}\label{spacesect}
Let $\mathbb{X}$ be a non-empty set. A \emph{quasi--metric} on $\mathbb{X}$ is a function $\rho :\mathbb{X}\times\mathbb{X} \to \mathbb{R}$ that satisfies the following properties: For all $x,y, z\in\mathbb{X}$,
\begin{enumerate}
\item $\rho(x,y)\ge 0$, 
\item $\rho(x,y)=0$ if and only if $x=y$, 
\item $\rho(x,y)=\rho(y,x)$, 
\item there exists a constant $\kappa_1\ge 1$ such that
\begin{equation}\label{quasimetricdef}
\rho(x,y)\le \kappa_1\{\rho(x,z)+\rho(z,y)\}, \qquad x, y, z\in\mathbb{X}.
\end{equation}
\end{enumerate}
For example, the geodesic distance on a Riemannian manifold  $\mathbb{X}$  is a quasi--metric.

The quasi--metric $\rho$ gives rise to a topology on $\mathbb{X}$, with 
$$
\{y\in\mathbb{X} : \rho(x,y)<r\}, \qquad x\in\mathbb{X},\ r>0.
$$
being a basis for the topology. In the sequel, we will write
$$
\mathbb{B}(x,r) = \{y\in\mathbb{X} : \rho(x,y)\le r\}, \ \Delta(x,r)=\mathbb{X}\setminus \mathbb{B}(x,r), \qquad x\in \mathbb{X}, \ r>0.
$$

In remainder of this paper, let $\mu^*$ be a fixed probability measure on $\mathbb{X}$. We fix a non-decreasing sequence $\{\lambda_k\}_{k=0}^\infty$ of nonnegative numbers such that $\lambda_0=0$, and $\lambda_k\uparrow \infty$ as $k\to\infty$. Also, we fix a system of continuous, bounded, and integrable functions $\{\phi_k\}_{k=0}^\infty$, orthonormal with respect to $\mu^*$; namely, for all nonnegative integers $j, k$,
\begin{equation}\label{orthonormality}
\int_\mathbb{X} \phi_k(x)\phi_j(x)d\mu^*(x) =\left\{\begin{array}{ll}
1, &\mbox{if $j=k$,}\\
0, &\mbox{otherwise.}
\end{array}\right.
\end{equation}
We will assume  that $\phi_0(x)=1$ for all $x\in\mathbb{X}$. 

For example, in the case of  a compact Riemannian manifold $\mathbb{X}$, we may take the (normalized) volume measure on $\mathbb{X}$ to be $\mu^*$, and take $\phi_k$'s to be the eigenfunctions of the Laplace--Beltrami operator on $\mathbb{X}$, corresponding to the eigenvalues $-\lambda_k^2$. 
If a different measure is assumed, then we may need to consider differential operators other than the Laplace--Beltrami operator. 
Also, it is sometimes not necessary to use the exact eigenvalues of such operators. 
For example, in the case when $\mathbb{X}$ is the unit sphere embedded in $\mathbb{R}^3$, the eigenvalues of the (negative) Laplace--Beltrami operator are given by $\sqrt{k(k+1)}$. 
The analysis is sometimes easier if we use $k$ instead. 
In general, while the exact eigenvalues might be hard to compute, an asymptotic expression is often available. 
While these considerations motivate our definitions, we observe that we are considering a very general scenario with quasi--metric measure spaces, where differential operators are not defined.
Nevertheless, we will refer to each $\phi_k$ as an \emph{eigenfunction} corresponding to the \emph{eigenvalue} $\lambda_k$, even though they are not necessarily obtained from an eigen-decomposition of any predefined differential or integral operator.

In our context, the role of polynomials will be played by diffusion polynomials, which are finite linear combinations of $\{\phi_j\}$. In particular, an element of
$$
\Pi_n :=\mathsf{span}\{\phi_j : \lambda_j <n\}
$$
will be called a diffusion polynomial of degree $<n$.

For reasons explained in the introduction, we will formulate our assumptions in terms of a formal heat kernel. The \emph{heat kernel} on $\mathbb{X}$ is defined formally by
\begin{equation}\label{heatkerndef}
K_t(x,y)=\sum_{k=0}^\infty \exp(-\lambda_k^2t)\phi_k(x){\phi_k(y)}, \qquad x, y\in \mathbb{X}, \ t>0.
\end{equation}
Although $K_t$ satisfies the semigroup property, and in light of the fact that $\lambda_0=0$, $\phi_0(x)\equiv 1$, we have formally
\begin{equation}\label{heatkernint}
\int_\mathbb{X} K_t(x,y)d\mu^*(y) =1, \qquad x\in\mathbb{X},
\end{equation}
yet $K_t$ may  not be the heat kernel in the classical sense. In particular, we need not assume $K_t$ to be nonnegative.

\begin{definition}\label{ddrdef}
The system $\Xi=(\mathbb{X},\rho,\mu^*,\{\lambda_k\}_{k=0}^\infty,\{\phi_k\}_{k=0}^\infty)$) is called a \textbf{data-defined space} if each of the following conditions are satisfied.
\begin{enumerate}
\item For each $x\in\mathbb{X}$ and $r>0$, the ball $\mathbb{B}(x,r)$ is compact.
\item There exist $q>0$ and $\kappa_2>0$ such that the following power growth bound condition holds:
\begin{equation}\label{ballmeasurecond}
\mu^*(\mathbb{B}(x,r))=\mu^*\left(\{y\in\mathbb{X} : \rho(x,y)<r\}\right) \le \kappa_2r^q, \qquad x\in\mathbb{X}, \ r>0.
\end{equation}
\item The series defining $K_t(x,y)$ converges for every $t\in (0,1]$ and $x,y\in\mathbb{X}$. Further, with $q$ as above, there exist $\kappa_3, \kappa_4>0$ such that the following Gaussian upper bound holds:
\begin{equation}\label{gaussianbd}
|K_t(x,y)|\le \kappa_3t^{-q/2}\exp\left(-\kappa_4\frac{\rho(x,y)^2}{t}\right), \qquad x, y\in \mathbb{X},\ 0<t\le 1.
\end{equation} 
\end{enumerate}
\end{definition}

There is a great deal of discussion in the literature on the validity of the  conditions in the above definition and their relationship with many other objects related to the quasi--metric space in question, (cf. for example,  \cite{davies1997,  grigor1999estimates, grigor2006heat, grigorlyan2heat}). 
In particular, it is shown in \cite[Section~5.5]{davies1997} that all the conditions defining a data-defined space are satisfied in the case of any complete, connected Riemannian manifold with non--negative Ricci curvature. 
It is shown in  \cite{kordyukov1991p}  that  our assumption on the heat kernel is valid in the case when $\mathbb{X}$ is a complete Riemannian manifold with bounded geometry, and $\{-\lambda_j^2\}$, respectively $\{\phi_j\}$, are eigenvalues, respectively eigenfunctions, for a uniformly elliptic second order differential operator satisfying certain technical conditions.

The bounds on the heat kernel are closely connected with the measures of the balls $\mathbb{B}(x,r)$. For example, using \eqref{gaussianbd}, Lemma~\ref{criticallemma} below, and the fact that
$$
\int_{\mathbb{X}} |K_t(x,y)|d\mu^*(y)\ge \int_{\mathbb{X}} K_t(x,y)d\mu^*(y)=1, \qquad x\in \mathbb{X},
$$
it is not difficult to deduce as in \cite{grigorlyan2heat} that
\begin{equation}\label{ballmeasurelowbd}
\mu^*(\mathbb{B}(x,r))\ge cr^q, \qquad 0<r\le 1.
\end{equation}
In many of the examples cited above, the kernel $K_t$ also satisfies a lower bound to match the upper bound in \eqref{gaussianbd}. In this case, Grigory\'an \cite{grigorlyan2heat} has also shown that \eqref{ballmeasurecond} is satisfied for $0<r<1$. 

We remark that the estimates  \eqref{ballmeasurecond} and \eqref{ballmeasurelowbd} together imply that $\mu^*$ satisfies the homogeneity condition
\begin{equation}\label{doublingcond}
\mu^*(\mathbb{B}(x,R))\le c_1(R/r)^q\mu^*(\mathbb{B}(x,r)), \qquad x\in\mathbb{X},\  r\in (0,1],\ R>0,
\end{equation}
where $c_1>0$ is a suitable constant.

In the sequel,  we assume that $\Xi$ is a data-defined space, and make the following convention.

\noindent\textbf{Constant convention:}\\
\textit{In the sequel, the symbols $c, c_1,\cdots$ will denote positive constants depending only on $\mathbb{X}$, $\rho$, $\mu^*$, $\kappa_1,\cdots,\kappa_5$,   and other similar fixed quantities such as the parameters denoting the various spaces. They will not depend upon the systems $\{\phi_k\}$, $\{\lambda_k\}$  by themselves, except through the quantities mentioned above. On occasions when we need to have the constants depend upon additional variables, these will be listed explicitly.  Their values may be different at different occurences, even within a single formula. The notation $A\sim B$ will mean $c_1A\le B\le c_2A$.
}
\subsection{Measures}\label{meassect}
In this paper, it is necessary to consider a sequence of sets $\mathcal{C}_n=\{x_{1,n},\cdots,x_{M_n,n}\}$, and the quadrature weights  $w_{k,n}$, leading to sums of the form
$$
\sum_{1\le k\le M_n\atop x_{k,n}\in B}w_{k,n}f(x_{k,n}),
$$
where $B\subseteq\mathbb{X}$. 
The precise locations of the points of $\mathcal{C}_n$, or the numbers $w_{k,n}$, or even the numbers $M_n$ will play no role in our theoretical development. Therefore, we find it convenient to use a sequence of measures to abbreviate sums like the above. Accordingly,  In this subsection, we review some measure theoretical notation and definitions.

 If $\nu$ is a (signed)  measure defined on a sigma algebra $\mathfrak{M}$ of $\mathbb{X}$, its total variation measure $|\nu|$ is defined by
$$
|\nu|(B)=\sup\sum_{k=1}^\infty|\nu(U_k)|,
$$
where the supremum is taken over all countable partitions $\{U_k\}\subseteq \mathfrak{M}$ of $B$. Here, the quantity $|\nu|(\mathbb{X})$ is called the total variation of $\nu$. If $\nu$ is a signed measure, then its total variation is always finite. If $\nu$ is a positive measure, it is said to be of bounded variation if its total variation is finite.
The measure $\nu$ is said to be complete if for any  $B\in \mathfrak{M}$ with $|\nu|(B)=0$ and any subset $A\subseteq B$, $A\in \mathfrak{M}$ and $|\nu|(A)=0$.  Since any measure can be extended to a complete measure by suitably enlarging the underlying sigma algebra,  we will assume in the sequel that all the measures to be introduced in this paper are complete. 

If $\mathcal{C}\subseteq\mathbb{X}$ is a finite set, the measure $\nu$ that associates with each $x\in\mathcal{C}$ the mass $w_x$,  is defined by 
$$
\nu(B)=\sum_{x\in B} w_x.
$$
for  subsets  $B\subseteq \mathbb{X}$. Obviously, the total variation of the measure $\nu$ is given by
$$
|\nu|(B)=\sum_{x\in B} |w_x|, \qquad B\subseteq \mathbb{X}.
$$
If $f : \mathcal{C}\to\mathbb{C}$, then for $B\subseteq\mathbb{X}$,
$$
\int_B fd\nu =\sum_{x\in\mathcal{C}\cap B}w_xf(x).
$$
Thus, in the example at the start of this subsection, if $\nu_n$ is the measure that associates the mass $w_{k,n}$ with $x_{k,n}$ for $k=1,\cdots, M_n$, then we have a concise notation
$$
\sum_{1\le k\le M_n\atop x_{k,n}\in B}w_{k,n}f(x_{k,n})=\int_B fd\nu_n \left(=\int_B f(x)d\nu_n(x)\right).
$$
 
 In the sequel, we will assume that every measure introduced in this paper is a complete, sigma finite, Borel measure; i.e.,
 the sigma algebra $\mathfrak{M}$ on which it is defined contains all Borel subsets of $\mathbb{X}$. In the rest of this paper, rather than stating that $\nu$ is defined on $\mathfrak{M}$, we will follow the usual convention of referring to members of $\mathfrak{M}$ as $\nu$-measurable sets without mentioning the sigma algebra explicitly.
\subsection{Smoothness Classes}\label{smoothsect}
 If $B\subseteq \mathbb{X}$ is $\nu$-measurable, and $f : B\to \mathbb{C}$ is a $\nu$-measurable function, we will write
$$
\|f\|_{\nu;B,p}:=\left\{\begin{array}{ll}
\displaystyle \left\{\int_B|f(x)|^pd|\nu|(x)\right\}^{1/p}, & \mbox{ if $1\le p<\infty$,}\\
\displaystyle|\nu|-\esssup_{x\in B}|f(x)|, &\mbox{ if $p=\infty$.}
\end{array}\right.
$$
We will write $L^p(\nu;B)$ to denote the class of all $\nu$--measurable functions $f$ for which $\|f\|_{\nu;B,p}<\infty$, where two functions are considered equal if they are equal $|\nu|$--almost everywhere. We will omit the mention of $\nu$ if $\nu=\mu^*$ and that of $B$ if $B=\mathbb{X}$. Thus, $L^p=L^p(\mu^*;\mathbb{X})$.
The $L^p$ closure of the set of all diffusion polynomials will be denoted by $X^p$.
 For $1\le p\le\infty$, we define $p'=p/(p-1)$ with the usual understanding that $1'=\infty$, $\infty'=1$. 

In the absence of a differentiability structure on $\mathbb{X}$, perhaps, the easiest way to define a Bessel potential space is the following. 
If $f_1\in L^p$, $f_2\in L^{p'}$ then
$$
\langle f_1, f_2\rangle:=\int_\mathbb{X} f_1(x)f_2(x)d\mu^*(x).
$$
In particular, we write
$$
\hat{f}(k)=\langle f, \phi_k\rangle, \qquad k=0,1,\cdots.
$$
For $r>0$, the pseudo--differential operator $\Delta^r$ is defined formally by
$$
\widehat{\Delta^r f}(k)=(\lambda_k+1)^r\hat{f}(k), \qquad k=0,1,\cdots.
$$
The class of all $f\in X^p$ for which there exists $\Delta^r f\in X^p$ with $\widehat{\Delta^r f}(k)$ as above is denoted by $W_r^p$. 
This definition is sometimes abbreviated in the form
$$
W_r^p=\left\{f\in X^p : \left\|\sum_k (\lambda_k+1)^r\hat{f}(k)\phi_k\right\|_p<\infty\right\}.
$$
However, since the series expansion need not converge in the $L^p$ norm, we prefer the distributional definition as we have given.

While the papers \cite{brauchart2014qmc, 
brauchart2015covering, brandolini2014quadrature, ehler_covering_radius2016, ehler_qmc_grassman2016} all deal with the spaces which we have denoted by $W_r^p$, we find it easier to consider a larger class, $H^p_\gamma$, defined as follows.
If $f\in X^p$, $r>0$, we define  a $K$-functional for $\delta>0$ by
\begin{equation}\label{kfuncdef}
\omega_r(p;f,\delta):=\inf\{\|f-f_1\|_p +\delta^r\|\Delta^rf_1\|_{p}\ :\ f_1\in W^p_r\}.
\end{equation}
If $\gamma>0$, we choose  $r>\gamma$, and define the smoothness class $H^p_\gamma$ to be the class of all $f\in X^p$ such that 
\begin{equation}\label{hpgammadef}
\|f\|_{H^p_\gamma}:=\|f\|_p+\sup_{\delta\in (0,1]}\frac{\omega_r(p;f,\delta)}{\delta^\gamma} <\infty.
\end{equation}

For example, if $\mathbb{X}=\mathbb{R}/(2\pi\mathbb{Z})$, $\mu^*$ is the arc measure on $\mathbb{X}$, $\{\phi_k\}$'s are the trigonometric monomials $\{1, \cos(k\circ), \sin(k\circ)\}_{k=1}^\infty$, and the eigenvalue $\lambda_k$ corresponding to $\cos(k\circ)$, $\sin(k\circ)$ is $|k|$, then the class $W_2^\infty$ is  the class of all twice continuously differentiable functions, while the class $H_2^\infty$ includes $f(x)=|\sin x|$.
The importance of the spaces $H_\gamma^p$ is well known in approximation theory \cite{devlorbk}.  We now describe the connection with approximation theory in our context.

If $f\in L^p$, $W\subseteq L^p$, we define
$$
\mathsf{dist}(p;f,W):=\inf_{P\in W}\|f-P\|_p.
$$

The following theorem is shown in \cite[Theorem~2.1]{mauropap} (where a different notation is used).
\begin{proposition}\label{equivprop}
Let $f\in X^p$. Then
\begin{equation}\label{equiv}
\|f\|_{H^p_\gamma} \sim \|f\|_p+\sup_{n>0}n^\gamma\mathsf{dist}(p; f, \Pi_n).
\end{equation}
\end{proposition}
In particular, different values of $r>\gamma$ give rise to the same smoothness class with equivalent norms (cf. \cite{devlorbk}). We note that $W^p_r\subset H^p_r$ for every  $r>0$.

\section{Main Results}\label{mainsect}
In this paper, we wish to state our theorems without the requirement that the quadrature formulas have positive weights, let alone equal weights. A substitute for this requirement is the notion of regularity (sometimes called continuity) condition. The space of all signed (or positive), complete, sigma finite, Borel measures on $\mathbb{X}$ will be denoted by $\mathcal{M}$.

\begin{definition}\label{absregularmeasuredef}
Let $d>0$. A  measure $\nu\in \mathcal{M}$  will be called \textbf{$ d$--regular} if
\begin{equation}\label{regulardef}
|\nu|(\mathbb{B}(x,d))\le cd^q, \qquad x\in\mathbb{X}.
\end{equation}
The infimum of all constants $c$ which work in \eqref{regulardef} will be denoted by $|\!|\!|\nu|\!|\!|_{R,d}$, and the class of all  $d$--regular measures will be denoted by $\mathcal{R}_d$. 
\end{definition}
For example, $\mu^*$ itself is in  ${\cal R}_d$ with $|\!|\!|\mu^*|\!|\!|_{R,d}\le \kappa_2$ for \emph{every} $d>0$ (cf. \eqref{ballmeasurecond}).  If $\mathcal{C}\subset \mathbb{X}$, we define the mesh norm $\delta(\mathcal{C})$ (also known as fill distance, covering radius, density content, etc.) and minimal separation $\eta(\mathcal{C})$ by
\begin{equation}\label{meshnormdef}
\delta(\mathcal{C})=\sup_{x\in\mathbb{X}}\inf_{y\in \mathcal{C}}\rho(x,y), \qquad \eta(\mathcal{C})=\inf_{x, y\in \mathcal{C}, \ x\not=y}\rho(x,y).
\end{equation}
It is easy to verify that if $\mathcal{C}$ is finite, the measure that associates the mass $\eta(\mathcal{C})^q$ with each point of $\mathcal{C}$ is $\eta(\mathcal{C})$-regular (\cite[Lemma~5.3]{eignet}). 

\begin{definition}\label{quadmeasdef}
Let $n\ge 1$. A measure $\nu\in\mathcal{M}$ is called a \textbf{quadrature measure of order $n$} if
\begin{equation}\label{quadrature}
\int_\mathbb{X} Pd\nu=\int_\mathbb{X} Pd\mu^*, \qquad P\in\Pi_n.
\end{equation}
An \textbf{MZ (Marcinkiewicz-Zygmund) quadrature measure} of order $n$ is a quadrature measure $\nu$ of order $n$ for which $|\!|\!|\nu|\!|\!|_{R,1/n}<\infty$.
\end{definition}

Our notion of approximate quadrature measures is formulated in the following definition. 
\begin{definition}\label{approxquaddef}
Let $\aleph=\{\nu_n\}_{n=1}^\infty\subset\mathcal{M}$, $\gamma>0$, $1\le p\le \infty$. We say that $\aleph$ is a sequence of \textbf{approximate quadrature measures of class $\mathcal{A}(\gamma,p)$} if each of the following conditions hold.
\begin{enumerate}
\item 
\begin{equation}\label{ubv}
 \sup_{n\ge 1}|\nu_n|(\mathbb{X})<\infty.
 \end{equation}
\item
\begin{equation}\label{ureg}
 \sup_{n\ge 1}|\!|\!|\nu_n|\!|\!|_{R,1/n} <\infty.
 \end{equation}
\item  For  $n\ge 1$,
\begin{equation}\label{approxquadpoly}
\left|\int_\mathbb{X} Pd\mu^*-\int_\mathbb{X} Pd\nu_n\right|\le A\frac{\|P\|_{H_\gamma^p}}{n^\gamma}, \qquad P\in \Pi_n,
\end{equation}
for a positive constant $A$ independent of $P$, but possibly dependent on $\aleph$ in addition to the other fixed parameters.
\end{enumerate}
With an abuse of terminology, we will often say that $\nu$ is an approximate quadrature measure of order $n$ (and write $\nu\in\mathcal{A}(\gamma,p,n)$) to mean tacitly that it is a member of a sequence $\aleph$ of approximate quadrature measures for which \eqref{approxquadpoly} holds.
\end{definition}

Clearly, if each $\nu_n$ is a quadrature measure of order $n$, then \eqref{approxquadpoly} is satisfied for every $\gamma>0$ and $1\le p\le \infty$. 
In the case of a compact Riemannian manifold satisfying some additional conditions, the existence of quadrature measures based on scattered data that satisfy the other two conditions in the above definition are discussed in \cite{frankbern, modlpmz}. 
In particular, we have shown in \cite{modlpmz} that under certain additional conditions, a sequence $\aleph$, where each $\nu_n$ is a  positive quadrature measure of order $n$ necessarily satisfies the first two conditions in Definition~\ref{approxquaddef}. 
In Theorem~\ref{posimpliesreg} below, we will give the analogue of this result in the present context.

First, we wish to state a theorem reconciling the notion of approximate quadrature measures with the usual notion of worst case error estimates.

\begin{theorem}\label{maintheo}
Let $n\ge 1$, $1\le p\le \infty$, $\gamma>q/p$, $\nu$ be a $1/n$-regular measure   satisfying $|\nu|(\mathbb{X})<\infty$ and  \eqref{approxquadpoly}.  Then for every $f\in H_\gamma^p$,
\begin{equation}\label{approxquadfull}
\left|\int_\mathbb{X} fd\mu^*-\int_\mathbb{X} fd\nu\right|\le c\left(A+|\!|\!|\nu|\!|\!|_{R,1/n}^{1/p}(|\nu|(\mathbb{X}))^{1/p'}\right)\frac{\|f\|_{H_\gamma^p}}{n^\gamma}.
\end{equation}
\end{theorem}

For example, if $\mathcal{C}\subset\mathbb{X}$ is a finite set with $\eta(\mathcal{C})\sim |\mathcal{C}|^{-1/q}$, and $\nu$ is the measure that associates the mass $|\mathcal{C}|^{-1}$ with each point of $\mathcal{C}$, then our notion of approximate quadrature measures generalizes the notion of approximate QMC designs in \cite{brauchart2014qmc, brauchart2015covering, ehler_covering_radius2016}. 
Since an MZ quadrature measure of order $n$ on a compact Riemannian manifold is in $\mathcal{A}(\gamma, p, n)$, Theorem~\ref{maintheo} generalizes essentially \cite[Theorem~5]{besovquadpap} (for the spaces denoted there by $B^\gamma_{p,\infty}$, which are our $H_\gamma^p$) as well as \cite[Asssertions~(B), (C)]{brandolini2014quadrature} (except that we consider the larger smoothness class than defined directly with the Bessel potentials). 
We note finally that together with Proposition~\ref{mzequivprop}, Theorem~\ref{maintheo} implies that if $\nu\in\mathcal{A}(\gamma,p,n)$ then $\nu\in\mathcal{A}(\gamma,p, \alpha n)$ for any positive $\alpha>0$ (although the various constants will then depend upon $\alpha$).

Next, we demonstrate in Theorem~\ref{kernmintheo} below that a sequence of approximate quadrature measures can be constructed as solutions of certain optimization problems under certain additional conditions. 
These optimization problems involve certain kernels of the form $G(x,y)=\sum_k b(\lambda_k)\phi_k(x)\phi_k(y)$, where (intuitively) $b(\lambda_k)\sim (\lambda_k+1)^{-\beta}$ for some $\beta$. 
In the case of integrating functions in $L^2$, only the order of magnitude estimates on the coefficients $b(\lambda_k)$ play a role.
In the other spaces, this is not sufficient because of an absence of the Parseval identity. On the other hand, restricting ourselves to Bessel potentials is not always an option in the case of the data-defined spaces; there is generally no closed form formula for these. 
A middle ground is provided by the following definition  (\cite[Definition~2.3]{eignet}).

\begin{definition}\label{eigkerndef}
Let $\beta\in\mathbb{R}$. A function $b:\mathbb{R}\to\mathbb{R}$ will be called a mask of type $\beta$ if $b$ is an even, $S$ times continuously differentiable function such that for $t>0$, $b(t)=(1+t)^{-\beta}F_b(\log t)$ for some $F_b:\mathbb{R}\to\mathbb{R}$ such that  $|F_b^{(k)}(t)|\le c(b)$, $t\in\mathbb{R}$, $k=0,1,\cdots,S$,  and $F_b(t)\ge c_1(b)$, $t\in\mathbb{R}$.     A function $G:\mathbb{X}\times\mathbb{X}\to \mathbb{R}$ will be called  a kernel of type $\beta$ if it admits a formal expansion $G(x,y)=\sum_{j=0}^\infty b(\lambda_j)\phi_j(x)\phi_j(y)$ for some mask $b$ of type $\beta>0$. If we wish to specify the connection between $G$ and $b$, we will write $G(b;x,y)$ in place of $G$.
\end{definition}

The definition of a mask of type $\beta$ can be relaxed somewhat, for example, the various bounds on $F_b$ and its derivatives may only be assumed for sufficiently large values of $|t|$ rather than for all $t\in\mathbb{R}$. If this is the case, one can construct a new kernel by adding a suitable diffusion polynomial (of a fixed degree) to $G$, as is customary in the theory of radial basis functions, and obtain a kernel whose mask satisfies the definition given above. This does not add any new feature to our theory. Therefore, we assume the more restrictive definition as given above.

\begin{theorem}\label{kernmintheo}
Let $1\le p\le \infty$, $\beta> q/p$, $G$ be a kernel of type $\beta$ in the sense of Definition~\ref{eigkerndef}. For a  measure $\nu$, we denote
\begin{equation}\label{minexp}
M_p(\nu)=\left\|\int_\mathbb{X} G(x, \circ)d\nu(x)-\int_\mathbb{X} G(x, \circ)d\mu^*(x)\right\|_{p'}.
\end{equation}
Let $n>0$, $K$ be any compact subset of  measures  such that 
$\sup_{\nu\in K}|\nu|(\mathbb{X})\le c$ and\\ $\sup_{\nu\in K}|\!|\!|\nu|\!|\!|_{R,2^{-n}}\le c$. If there exists a quadrature measure $\nu^*$ of order $2^n$ in $K$, and $\nu^\#\in K$ satisfies
\begin{equation}\label{mincond}
M_p(\nu^\#) \le c\inf_{\nu\in K}M_p(\nu).
\end{equation}
Then $\nu^\#$ satisfies \eqref{approxquadpoly} for every $\gamma$, $0<\gamma<\beta$, and in particular, $\nu^\#\in\mathcal{A}(\gamma,p,2^n)$ for each such $\gamma$. 
\end{theorem}

The main purpose of Theorem~\ref{kernmintheo} is to suggest a way to construct an approximate quadrature measure $\nu^\#$ by solving a minimization problem involving different compact sets as appropriate for the applications. We illustrate by a few examples.

In some applications, the interest is in the choice of the points, stipulating the quadrature weights. 
For example, in the context of the sphere, the existence of points yielding equal weights quadrature is now known \cite{bondarenko2010optimal}. 
So, one may stipulate equal weights and seek an explicit construction for the points to yield equal weights approximate quadrature measures as in \cite{brauchart2014qmc}. 
In this case, $K$ can be chosen to be the set of all equal weight measures supported at points on $\mathbb{S}^q$,
the compactness of this set following from that of the tensor product of the spheres.
In the context of machine learning, the points cannot be chosen and the interest is in finding the weights of an approximate quadrature measure.
 The existence of positive quadrature formulas (necessarily satisfying the required regularity conditions)  are known in the case of a manifold, subject to certain conditions on the points and the manifold in question \cite{frankbern, modlpmz}. 
In this case, the set $K$ can be taken to be that of all positive unit measures  supported at these points. 
In general, if $\mathbb{X}$ is compact, Tchakaloff's theorem shows that one could seek to obtain   approximate quadrature measures computationally by minimizing the quantity $M_p(\nu)$ over all positive unit measures $\nu$ supported on a number of points equal to the dimension of $\Pi_n$ for each $n$.

We make some remarks regarding the computational aspects. The problem of finding exact quadrature weights is the problem of solving an over-determined system of equations involving the eigenfunctions.
Since these eigenfunctions are themselves known only approximately from the data, it is desirable to work directly with a kernel. In the case of $L^2$, the minimization problem to find non-negative weights is the problem of minimizing
$$
\sum_{j,\ell} w_jw_\ell G^*(x_j,x_\ell), \qquad x_j,x_\ell\in \mathcal{C}\subset\mathcal{P}
$$ 
over all non-negative $w_j$ with $\sum_j w_j =1$, where
$$
G^*(x,y)=\sum_{k=1}^\infty b(\lambda_k)^2\phi_k(x)\phi_k(y).
$$
Thus, the optimization problem involves only the training data. In the context of semi--supervised learning, a large point cloud $\mathcal{P}$ is given, but the labels are available only at a much smaller subset $\mathcal{C}\subset \mathcal{P}$. 
In this case, we need to seek approximate quadrature measures supported only on $\mathcal{C}$, but may use the entire set $\mathcal{P}$ to compute these. Thus,
in order to apply Theorem~\ref{kernmintheo}, we may choose $K$ to be the set of all measures with total variation $\le 1$, supported on $\mathcal{C}$, and estimate the necessary norm expressions using
the entire point cloud $\mathcal{P}$.
We observe that we have not stipulated a precise solution of an optimization problem, only a solution in the sense of \eqref{mincond}. 
In the context of data-defined spaces, the data-based kernels themselves are only approximations of the actual kernels, and hence, Theorem~\ref{kernmintheo} provides a theoretical justification for using algorithms to find sub-optimal solutions to the minimization problem in order to find approximate quadrature formulas. 

We end this discussion by observing  a proposition used in the proof of Theorem~\ref{kernmintheo}.

\begin{proposition}\label{kernminprop}
Let $1\le p\le \infty$, $\beta> q/p$, $G$ be a kernel of type $\beta$ in the sense of Definition~\ref{eigkerndef}. 
If $n>0$, $\nu^\#\in\mathcal{M}$, and $M_p(\nu^\#) \le \tilde{A}2^{-n\beta}$, then $\nu^\#$ satisfies \eqref{approxquadpoly} with $2^n$ replacing $n$ and $A=c\tilde{A}$. 

\end{proposition}

Next, we consider the density of the supports of the measures in a sequence of approximate quadrature measures. It is observed in \cite{brauchart2015covering, ehler_covering_radius2016} that there is a close connection between approximate QMC designs and the (asymptotically) optimal covering radii of the supports of these designs. 
The definition in these papers is given in terms of the number of points in the support. Typically, for a QMC design of order $n$, this number is $\sim n^q$. 
Therefore, it is easy to interpret this definition in terms of the mesh norm of the support of a QMC design of order $n$ being $\sim 1/n$. 
Our definition of an approximate quadrature measure sequence does not require the measures involved to be finitely supported. Therefore, the correct analogue of this definition seems to be the assertion that every ball of radius $\sim 1/n$ should intersect the support of an approximate quadrature measure of order $n$. The following 
Theorem~\ref{covertheo} gives a sharper version of this sentiment.

\begin{theorem}\label{covertheo}
Let $\gamma>0$,  $1\le p\le\infty$ and $\aleph=\{\nu_n\}$ be a sequence in $\mathcal{A}(\gamma, p)$. Let $\tilde{p}=1+q/(\gamma p')$.
 Then there exists a constant $C_1$ such that for $n\ge c$,
\begin{equation}\label{nulowbd}
|\nu_n|(\mathbb{B}(x,C_1/n^{1/\tilde{p}}))\ge c_1n^{-q/\tilde{p}}, \qquad x\in\mathbb{X}.
\end{equation} 
In particular,  if $\aleph$ is a sequence of  approximate quadrature  measures of 
class $\mathcal{A}(\gamma, 1)$, then
\begin{equation}\label{nulowbdprec}
|\nu_n|(\mathbb{B}(x,C_1/n))\sim c_1n^{-q}, \qquad x\in\mathbb{X}.
\end{equation} 
\end{theorem}

The condition \eqref{nulowbdprec} ensures that the support of the measure $\nu\in\mathcal{A}(\gamma,1,n)$ must contain at least $cn^q$ points. 
In several papers, including \cite{brandolini2014quadrature}, this fact was used to show the existence of a function in $H_\gamma^p$ for which there holds a lower bound corresponding to the upper bound in \eqref{approxquadpoly}. 
The construction of the ``bad function'' in these papers involves the notion of infinitely differentiable functions and their pointwise defined derivatives.
Since we are not assuming any differentiability structure on $\mathbb{X}$, so the notion of a $C^\infty$ function in the sense of derivatives is not possible in this context.  

Finally, we note in this connection that the papers \cite{brauchart2014qmc, brauchart2015covering, brandolini2014quadrature} deal exclusively with non-negative weights. 
The definition of approximate QMC designs in these papers does not require a regularity condition as we have done.  We show  under some extra conditions that if $\aleph$ is a sequence of positive measures such that for  each $n\ge 1$, $\nu_n$ satisfies  \eqref{approxquadpoly} with $p=1$ and some $\gamma>0$, then $\aleph\in \mathcal{A}(\gamma,1)$. 

For this purpose, we need to overcome a technical hurdle. In the case of the sphere, the product of two spherical polynomials of degree $<n$ is another spherical polynomial of degree $<2n$. Although a similar fact is valid in many other manifolds, and has been proved in \cite{geller2011band, modlpmz} in the context of eigenfunctions of very general elliptic differential operators on certain manifolds, we need to make an explicit assumption in the context of the present paper, where we do not assume any differentiability structure.\\

\noindent
\textsc{Product assumption:}\\
\textit{For  $A, N>0$, let
\begin{equation}\label{gammanormdef}
\epsilon_{A,N}:=\sup_{\lambda_j,\lambda_k\le N}\mathsf{dist}(\infty;\phi_j\phi_k,\Pi_{AN}).
\end{equation}
We assume that there exists $A^*\ge 2$ with the following property: for \textbf{every} $R>0$, $\displaystyle\lim_{N\to\infty}N^R\epsilon_{A^*,N}= 0$.}

 In the sequel, for any $H:\mathbb{R}\to\mathbb{R}$, we define formally
\begin{equation}\label{phikerndef}
\Phi_N(H;x,y):=\sum_{j=0}^\infty H(\lambda_j/N)\phi_j(x)\phi_j(y), \qquad x,y\in\mathbb{X}, \ N>0.
\end{equation}
In the remainder of this paper, we will fix an infinitely differentiable, even function $h :\mathbb{R}\to \mathbb{R}$ such that $h(t)=1$ if $|t|\le 1/2$, $h(t)=0$ if $|t|\ge 1$, and $h$ is non-increasing on $[1/2,1]$. The mention of this function will be usually omitted from the notation; e.g., we write $\Phi_n(x,y)$ in place of $\Phi_n(h;x,y)$.

\begin{theorem}\label{posimpliesreg}
Let $n>0$, $1\le p\le\infty$, $\nu$ be a positive measure satisfying \eqref{approxquadpoly} for some $\gamma>0$. We assume that the product assumption holds, and that in addition the following inequality holds: there exists $\beta>0$ such that
\begin{equation}\label{phinlowbd}
\min_{y\in \mathbb{B}(x,\beta/m)}|\Phi_m(x,y)| \ge cm^q, \qquad x\in \mathbb{X}, \ m\ge 1.
\end{equation}
Then 
\begin{equation}\label{regmeasbd}
\nu(\mathbb{B}(x,1/n))\le cn^{-q/p}, \qquad x\in \mathbb{X}.
\end{equation}
In particular,  if $p=1$ then $|\nu|(\mathbb{X})\le c$, $\nu\in\mathcal{R}_{1/n}$, and $|\!|\!|\nu|\!|\!|_{R,1/n}\le c$.
\end{theorem}

The condition \eqref{phinlowbd} is proved in \cite[Lemma~7.3]{modlpmz} in the case of compact Riemannian manifolds satisfying a gradient condition on the heat kernel (in particular, the spaces considered in the above cited papers).

\section{Preparatory Results}\label{prepsect}

In this section, we collect together some known results. We will supply the proofs for the sake of completeness when they are not too complicated.

\subsection{Results on Measures}\label{meas_res_sect}

The following proposition (cf. \cite[Proposition~5.6]{modlpmz}) reconciles different notions of regularity condition on measures defined in our papers.

\begin{proposition}\label{mzequivprop}
Let $d\in (0,1]$, $\nu\in {\cal M}$. \\
{\rm (a)} If $\nu$ is $d$--regular, then for each $r>0$ and $x\in\mathbb{X}$,
\begin{equation}\label{regreconcile}
|\nu|(\mathbb{B}(x,r))\le c|\!|\!|\nu|\!|\!|_{R,d}\ \mu^*(\mathbb{B}(x,c(r+d)))\le  c_1|\!|\!|\nu|\!|\!|_{R,d}(r+d)^q.
\end{equation}
Conversely, if for some $A>0$, $|\nu|(\mathbb{B}(x,r))\le A(r+d)^q$ or each $r>0$ and $x\in\mathbb{X}$, then $\nu$ is $d$--regular, and $|\!|\!|\nu|\!|\!|_{R,d}\le 2^q A$.\\
{\rm (b)} For each $\alpha>0$, 
\begin{equation}\label{regequiv}
|\!|\!|\nu|\!|\!|_{R,\alpha d}\le c_1(1+1/\alpha)^q |\!|\!|\nu|\!|\!|_{R,d}\le c_1^2(1+1/\alpha)^q(\alpha+1)^q|\!|\!|\nu|\!|\!|_{R,\alpha d},
\end{equation}
where $c_1$ is the constant appearing in \eqref{regreconcile}.\\
\end{proposition}

If $K\subseteq\mathbb{X}$ is a compact subset and $\epsilon>0$, we will say that a subset $\mathcal{C}\subseteq K$ is $\epsilon$--separated if $\rho(x,y)\ge \epsilon$ for every $x,y\in \mathcal{C}$, $x\not=y$. Since $K$ is compact, there exists a finite, maximal $\epsilon$--separated subset $\{x_1,\cdots,x_M\}$ of $K$. If $x\in K\setminus \cup_{k=1}^M \mathbb{B}(x_k,\epsilon)$, then $\{x,x_1,\cdots,x_M\}$ is a strictly larger $\epsilon$--separated subset of $K$. So, $K\subseteq 
\cup_{k=1}^M \mathbb{B}(x_k,\epsilon)$. Moreover, with $\kappa_1$ as in \eqref{quasimetricdef}, the balls $\mathbb{B}(x_k,\epsilon/(3\kappa_1))$ are mutually disjoint. \\

\noindent\textit{Proof of Proposition~\ref{mzequivprop}.} In the proof of part (a) only, let $\lambda>|\!|\!|\nu|\!|\!|_{R,d}$, $r>0$, $x\in\mathbb{X}$, and let $\{y_1,\cdots,y_N\}$ be a maximal $2d/3$--separated subset of $\mathbb{B}(x,r+2d/3)$. Then $\mathbb{B}(x,r)\subseteq \mathbb{B}(x,r+2d/3)\subseteq \cup_{j=1}^N \mathbb{B}(y_j,2d/3)$. So,
\begin{eqnarray*}
|\nu|(\mathbb{B}(x,r)) &\le& |\nu|(\mathbb{B}(x,r+2d/3))\le \sum_{j=1}^N |\nu|(\mathbb{B}(y_j,2d/3))\\
&\le& \sum_{j=1}^N |\nu|(\mathbb{B}(y_j,d)) \le \lambda Nd^q.
\end{eqnarray*}
The balls $\mathbb{B}(y_j,d/(3\kappa_1))$ are mutually disjoint, and $\cup_{j=1}^N \mathbb{B}(y_j,d/(3\kappa_1))\subseteq \mathbb{B}(x,c(r+d))$. In view of \eqref{ballmeasurelowbd}, $d^q\le c\mu^*(\mathbb{B}(y_j,d/(3\kappa_1)))$ for each $j$. So,
\begin{eqnarray*}
|\nu|(\mathbb{B}(x,r))&\le& \lambda Nd^q\le c\lambda\sum_{j=1}^N \mu^*(\mathbb{B}(y_j,d/(3\kappa_1)))=c\lambda\mu^*(\cup_{j=1}^N \mathbb{B}(y_j,d/(3\kappa_1)))\\
&\le& c\lambda \mu^*(\mathbb{B}(x,c(r+d))).
\end{eqnarray*}
Since $\lambda>|\!|\!|\nu|\!|\!|_{R,d}$ was arbitrary, this leads to the first inequality in \eqref{regreconcile}. The second inequality follows from \eqref{ballmeasurecond}. The converse statement is obvious. This completes the proof of part (a).

Using \eqref{regreconcile} with $\alpha d$ in place of $r$, we see that
$$
|\nu|(\mathbb{B}(x,\alpha d))\le c_1(\alpha+1)^qd^q|\!|\!|\nu|\!|\!|_{R,d}=c_1(1+1/\alpha)^q(\alpha d)^q|\!|\!|\nu|\!|\!|_{R,d}.
$$
This implies the first inequality in \eqref{regequiv}. The second inequality follows from the first, applied with $1/\alpha$ in place of $\alpha$.
\qed

Next, we prove a lemma (cf. \cite[Proposition~5.1]{eignet}) which captures many details of the proofs in Section~\ref{kern_res_sect}.

\begin{lemma}\label{criticallemma}
Let $\nu\in {\cal R}_d$,  $N>0$. If $g_1:[0,\infty)\to [0,\infty)$ is a nonincreasing function, then for any $N>0$, $r>0$, $x\in\mathbb{X}$,
\begin{equation}\label{g1ineq}
N^q\int_{\Delta(x,r)}g_1(N\rho(x,y))d|\nu|(y)\le c\frac{2^{q}(1+(d/r)^q)q}{1-2^{-q}}|\!|\!|\nu|\!|\!|_{R,d}\int_{rN/2}^\infty g_1(u)u^{q-1}du.
\end{equation}
\end{lemma}

\begin{proof}\ 
By replacing $\nu$ by $|\nu|/|\!|\!|\nu|\!|\!|_{R,d}$, we may assume that $\nu$ is  positive, and $|\!|\!|\nu|\!|\!|_{R,d}=1$.  Moreover, for  $r>0$, $\nu(\mathbb{B}(x,r))\le  c(1+(d/r)^q)r^q$. In this proof only, we will write $\mathbb{A}(x,t)=\{y\in\mathbb{X}\ :\ t< \rho(x,y)\le 2t\}$. We note that $\nu(\mathbb{A}(x,t))\le c2^q(1+(d/r)^q)t^q$, $t\ge r$, and 
$$
\int_{2^{R-1}}^{2^R}u^{q-1}du = \frac{1-2^{-q}}{q} 2^{Rq}.
$$
Since $g_1$ is nonincreasing, we have
\begin{eqnarray*}
\lefteqn{\int_{\Delta(x,r)}g_1(N\rho(x,y))d\nu(y)=\sum_{R=0}^\infty \int_{\mathbb{A}(x,2^{R}r)}g_1(N\rho(x,y))d\nu(y)}\\
&\le& \sum_{R=0}^\infty g_1(2^RrN)\nu(\mathbb{A}(x,2^{R}r)) \le c2^q(1+(d/r)^q)\sum_{R=0}^\infty g_1(2^RrN)(2^{R}r)^q\\
&\le&c\frac{2^{q}(1+(d/r)^q)q}{1-2^{-q}}r^q\sum_{R=0}^\infty \int_{2^{R-1}}^{2^R} g_1(urN)u^{q-1}du \\
&=& c\frac{2^{q}(1+(d/r)^q)q}{1-2^{-q}}r^q\int_{1/2}^\infty g_1(urN)u^{q-1}du\\
&=&c\frac{2^{q}(1+(d/r)^q)q}{1-2^{-q}}N^{-q}\int_{rN/2}^\infty g_1(v)v^{q-1}dv.
\end{eqnarray*}
This proves \eqref{g1ineq}.
\end{proof}

\subsection{Results on Kernels}\label{kern_res_sect}

In our theory, a fundamental role is played by the kernels defined formally in \eqref{phikerndef}:
\begin{equation}\label{phikerndefbis}
\Phi_N(H;x,y):=\sum_{j=0}^\infty H(\lambda_j/N)\phi_j(x)\phi_j(y), \qquad x,y\in\mathbb{X}, \ N>0.
\end{equation}
To describe the properties of this kernel, we introduce the notation
$$
\||H|\|_S:=\max_{0\le k\le S}\max_{x\in\mathbb{R}}|H^{(k)}(x)|.
$$
A basic and important property of these kernels is given in the following theorem.

\begin{theorem}\label{kernloctheo}
Let  $S>q$ be an integer, $H:\mathbb{R}\to \mathbb{R}$ be an even, $S$ times continuously differentiable, compactly supported function. 
 Then for every $x,y\in \mathbb{X}$, $N>0$,
\begin{equation}\label{kernlocest}
| \Phi_N(H;x,y)|\le \frac{cN^{q}\||H\||_S}{\max(1, (N\rho(x,y))^S)}.
\end{equation}
\end{theorem}

 Theorem~\ref{kernloctheo}  is proved in \cite{mauropap}, and more recently  in much greater generality in \cite[Theorem~4.3]{tauberian}.
 In \cite{mauropap}, Theorem~\ref{kernloctheo} was proved under the conditions that the so called finite speed of wave propagation holds, and the following \emph{spectral bounds} hold for the so called Christoffel (or spectral) function (defined by the sum expression in \eqref{christbd} below):
\begin{equation}\label{christbd}
 \sum_{\lambda_j< N}|\phi_j(x)|^2 \le cN^q, \qquad x\in\mathbb{X}, \ N>0.
\end{equation}
We have  proved in \cite[Theorem~4.1]{frankbern} that  \eqref{gaussianbd} with $y\not=x$ is equivalent to  the  finite speed of wave propagation. We have also shown in \cite[Proposition~4.1]{frankbern} and \cite[Lemma~5.2]{eignet} that \eqref{gaussianbd} with $y=x$ 
is equivalent to  \eqref{christbd}.

The following proposition follows easily from Lemma~\ref{criticallemma} and Theorem~\ref{kernloctheo}.

\begin{proposition}\label{criticalprop}
Let  $S$, $H$ be as in Theorem~\ref{kernloctheo}, $d>0$, $\nu\in\mathcal{R}_d$, and  $x\in\mathbb{X}$. \\
{\rm (a)} If  $r\ge  1/N$, then
\begin{equation}\label{phiintaway}
\int_{\Delta(x,r)}|\Phi_N(H;x,y)|d|\nu|(y) \le c(1+(dN)^q)(rN)^{-S+q}|\!|\!|\nu|\!|\!|_{R,d}\||H\||_S.
\end{equation}
{\rm (b)} We have
\begin{equation}\label{phiinttotal}
\int_\mathbb{X}|\Phi_N(H;x,y)|d|\nu|(y)\le c(1+(dN)^q)|\!|\!|\nu|\!|\!|_{R,d}\||H\||_S,
\end{equation} 
\begin{equation}\label{philpnorm}
\|\Phi_N(H;x,\circ)\|_{\nu;\mathbb{X},p} \le cN^{q/p'}(1+(dN)^q)^{1/p}|\!|\!|\nu|\!|\!|_{R,d}^{1/p}\||H\||_S,
\end{equation}
and
\begin{equation}\label{phiintnorm}
\left\|\int_\mathbb{X}|\Phi_N(H;\circ,y)|d|\nu|(y)\right\|_p \le c(1+(dN)^q)^{1/p'}|\!|\!|\nu|\!|\!|_{R,d}^{1/p'}(|\nu|(\mathbb{X}))^{1/p}\||H\||_S.
\end{equation}
\end{proposition}
\begin{proof}\ 
Without loss of generality, we assume that $\nu$ is a positive measure and assume also the normalizations $|\!|\!|\nu|\!|\!|_{R,d}=\||H\||_S=1$. Let $x\in\mathbb{X}$, $N>0$. For $r\ge 1/N$, $d/r\le dN$. In view of \eqref{kernlocest} and \eqref{g1ineq}, we have for $x\in \mathbb{X}$:
\begin{eqnarray*}
\int_{\Delta(x,r)}|\Phi_N(H;x,y)|d\nu(y)&\le& cN^q \int_{\Delta(x,r)}(N\rho(x,y))^{-S}d\nu(y)\\
&\le& 
c(1+(dN)^q)\int_{rN/2}^\infty v^{-S+q-1}dv\\
&\le& c(1+(dN)^q)(rN)^{-S+q}.
\end{eqnarray*}
This proves \eqref{phiintaway}.

Using \eqref{phiintaway} with $r=1/N$, we obtain that
\begin{equation}\label{pf1eqn1}
\int_{\Delta(x,1/N)}|\Phi_N(H;x,y)|d\nu(y)\le c(1+(dN)^q).
\end{equation}
We observe that in view of \eqref{kernlocest}, and the fact that $\nu(\mathbb{B}(x,1/N))\le c(1/N+d)^q \le cN^{-q}(1+(dN)^q)$,
$$
\int_{\mathbb{B}(x,1/N)}|\Phi_N(H;x,y)|d\nu(y)\le cN^q\nu(\mathbb{B}(x,1/N))\le c(1+(dN)^q).
$$
 Together with \eqref{pf1eqn1}, this leads to \eqref{phiinttotal}.

The estimate \eqref{philpnorm} follows from \eqref{kernlocest} in the case $p=\infty$, and from \eqref{phiinttotal} in the case $p=1$. For $1<p<\infty$, it follows from the convexity inequality
\begin{equation}\label{convexityineq}
\|F\|_{\nu;\mathbb{X},p}\le \|F\|_{\nu;\mathbb{X},\infty}^{1/p'}\|F\|_{\nu;\mathbb{X},1}^{1/p}.
\end{equation}
The estimate \eqref{phiintnorm} is the same as \eqref{phiinttotal} in the case when $p=\infty$. In addition, using \eqref{phiinttotal} with $\mu^*$ in place of $\nu$, $1/N$ in place of $d$, we obtain 
$$
\int_\mathbb{X} |\Phi_N(H;x,y)|d\mu^*(x) = \int_\mathbb{X} |\Phi_N(H;y,x)|d\mu^*(x) \le c.
$$
Therefore,
$$
\int_\mathbb{X} \int_\mathbb{X}|\Phi_N(H;x,y)|d|\nu|(y)d\mu^*(x)=\int_\mathbb{X} \int_\mathbb{X}|\Phi_N(H;x,y)|d\mu^*(x)d|\nu|(y)\le c|\nu|(\mathbb{X}).
$$
This proves \eqref{phiintnorm} in the case when $p=1$. The estimate in the general case follows from the cases $p=1,\infty$ and \eqref{convexityineq}.
\end{proof}

Next, we study some operators based on these kernels. 
If $\nu$ is any measure on $\mathbb{X}$ and $f\in L^p$, we may define formally
\begin{equation}\label{sigmaopdef}
\sigma_N(H;\nu; f,x):=\int_\mathbb{X} f(y)\Phi_N(H;x,y)d\nu(y). 
\end{equation}

The following is an immediate corollary of  Proposition~\ref{criticalprop}, used with $\mu^*$ in place of $\nu$, $d=0$.
\begin{corollary}\label{kerncor}
We have
\begin{equation}\label{kernbdest}
\sup_{x\in\mathbb{X}}\int_{\mathbb{X}}| \Phi_N(H;x,y)|d\mu^*(y) \le c\||H\||_S,
\end{equation}
and for every $1\le p\le \infty$ and $f\in L^p$, 
\begin{equation}\label{sigmaopbd}
\| \sigma_N(H;\mu^*;f)\|_p\le c\||H\||_S\|f\|_p.
\end{equation}
\end{corollary}

We recall that $h : \mathbb{R}\to\mathbb{R}$ denotes a fixed, infinitely differentiable,  and even function, nonincreasing on $[0,\infty)$, such that $h(t)=1$ if $|t|\le 1/2$ and $h(t)=0$ if $|t|\ge 1$. 
We  omit the mention of $h$ from the notation, and all constants $c, c_1,\cdots$ may depend upon $h$. 
 As before, we will omit the mention of $\nu$ if $\nu=\mu^*$ and that of $H$ if  $H=h$. Thus, $\Phi_N(x,y)=\Phi_N(h;x,y)$, and similarly  $\sigma_N(f,x)=\sigma_N(h;\mu^*;f,x)$, $\sigma_N(\nu;f,x)=\sigma_N(h;\nu;f,x)$. 
 The slight inconsistency is resolved by the fact that we use $\mu^*$, $\nu$, $\tilde\nu$ etc. to denote measures and  $h$, $g$, $b$, $H$, etc. to denote functions. We do not consider this to be a sufficiently important issue to complicate our notations. 

The following proposition gives the approximation properties of the kernels, and summarizes some important inequalities in approximation theory in this context. Different parts of this proposition are proved in \cite{mauropap, eignet}.

\begin{proposition}\label{approxprop}
Let $1\le p\le\infty$,  $N>0$, $r>0$. \\
{\rm (a)} For $f\in L^p$,
\begin{equation}\label{goodapprox}
\mathsf{dist}(p;f,\Pi_N)\le \|f-\sigma_N(f)\|_p\le c\mathsf{dist}(p;f,\Pi_{N/2}).
\end{equation}
{\rm (b)} If $f\in W^p_r$, then
\begin{equation}\label{polyfavard}
\mathsf{dist}(p;f,\Pi_N)\le \|f-\sigma_N(f)\|_p \le cN^{-r}\|\Delta^rf\|_p.
\end{equation}
{\rm (c)}
For  $P\in \Pi_N$,
\begin{equation}\label{pseudobern}
\|\Delta^rP\|_p \le cN^r\|P\|_p.
\end{equation}
{\rm (d)} For $f\in L^p$,
\begin{equation}\label{jacksonest}
\omega_r(p;f,1/N)\le \|f-\sigma_N(f)\|_p +N^{-r}\|\Delta^r\sigma_N(f)\|_p \le c\omega_r(p;f,1/N).
\end{equation}

\end{proposition}

\begin{proof}\ 
If $P\in\Pi_{N/2}$ is chosen so that $\|f-P\|_p\le 2\mathsf{dist}(p;f,\Pi_{N/2})$, then \eqref{sigmaopbd} implies that
$$
\|f-\sigma_N(f)\|_p =\|f-P-\sigma_N(f-P)\|_p \le c\|f-P\|_p\le c\mathsf{dist}(p;f,\Pi_{N/2}).
$$
This proves part (a). 

 The   parts (b) and (c) are proved in \cite[Theorem~6.1]{mauropap}.

Next, let $f_1$ be chosen so that $\|f-f_1\|_p+N^{-r}\|\Delta^rf_1\|_p \le 2\omega_r(p;f,1/N)$. Then using \eqref{sigmaopbd}, \eqref{polyfavard}, and \eqref{pseudobern}, we deduce that
\begin{eqnarray*}
\lefteqn{\|f-\sigma_N(f)\|_p +N^{-r}\|\Delta^r\sigma_N(f)\|_p}\\
& \le& \|f-f_1-\sigma_N(f-f_1)\|_p +\|f_1-\sigma_N(f_1)\|_p\\
&& \qquad\qquad +N^{-r}\left(\|\Delta^r\sigma_N(f-f_1)\|_p +\|\Delta^r\sigma_N(f_1)\|_p\right)\\
&\le& c\{\|f-f_1\|_p +N^{-r}\|\Delta^rf_1\|_p + \|\sigma_N(f-f_1)\|_p +N^{-r}\|\sigma_N(\Delta^rf_1)\|_p\}\\
&\le& c\{\|f-f_1\|_p +N^{-r}\|\Delta^rf_1\|_p\}\le c\omega_r(p;f,1/N).
\end{eqnarray*}
This proves \eqref{jacksonest}. 
\end{proof}

We note next a corollary of this proposition.

\begin{corollary}\label{sigmahnormcor}
Let $r>\gamma>0$, $\delta\in (0,1]$, $1\le p \le \infty$, $f\in X^p$, and $n\ge 1$. Then
\begin{equation}\label{sigmahnormineq}
\omega_r(p; \sigma_n(f), \delta) \le c\omega_r(p;f,\delta)\, \qquad\|\sigma_n(f)\|_{H_\gamma^p}\le c\|f\|_{H_\gamma^p}.
\end{equation}
\end{corollary}
\begin{proof}\ 
Let $N\ge 1$ be chosen such that $1/(2N)< \delta\le 1/N$. A comparison of the Fourier coefficients shows that
$$
\sigma_N(\sigma_n(f))=\sigma_n(\sigma_N(f)), \qquad \Delta^r(\sigma_N(\sigma_n(f)))=\sigma_n(\Delta^r(\sigma_N(f))).
$$
Consequently, using \eqref{jacksonest}, we conclude that
\begin{eqnarray*}
\omega_r(p;\sigma_n(f),\delta)&\le& \omega_r(p;\sigma_n(f),1/N)\\
&\le&  \|\sigma_n(f)-\sigma_N(\sigma_n(f))\|_p +\frac{1}{N^r}\|\Delta^r(\sigma_N(\sigma_n(f)))\|_p\\
&=& \|\sigma_n(f)-\sigma_n(\sigma_N(f))\|_p +\frac{1}{N^r}\|\sigma_n(\Delta^r(\sigma_N(f)))\|_p\\
&\le& c\{\|f-\sigma_N(f)\|_p +\frac{1}{N^r}\|\Delta^r(\sigma_N(f))\|_p\}\\
&\le& c\omega_r(p;f,1/N) \le c_1\omega_r(p;f, 1/(2N))\le c_1\omega_r(p;f,\delta).
\end{eqnarray*}
This proves the first inequality in \eqref{sigmahnormineq}. The second inequality is now immediate from the definitions.
\end{proof}

Next, we state another fundamental result, that characterizes the space $H_\gamma^p$ in terms of a series expansion of the functions.
In the sequel, we will write for $f\in X^1\cap X^\infty$, $x\in\mathbb{X}$,
$$
\tau_j(f,x)=\left\{\begin{array}{ll}
\sigma_1(f,x), &\mbox{ if $j=0$,}\\
\sigma_{2^j}(f,x)-\sigma_{2^{j-1}}(f,x), &\mbox{ if $j=1,2,\cdots$.}
\end{array}\right.
$$
The following lemma summarizes some relevant properties of these operators.

\begin{lemma}\label{taulemma}
Let $1\le p\le\infty$,  $f\in X^p$.\\
{\rm (a)} We have
\begin{equation}\label{lpseries}
f=\sum_{j=0}^\infty \tau_j(f),
\end{equation}
with convergence in the sense of $L^p$.\\
{\rm (b)} For each $j=2,3,\cdots$,
\begin{equation}\label{taubasicest}
\|\tau_j(f)\|_p\le c\mathsf{dist}(p;f,\Pi_{2^{j-2}})\le c\sum_{k=j-1}^\infty \|\tau_k(f)\|_p.
\end{equation}
In particular, if $\gamma>0$ and $f\in H_\gamma^p$, then
\begin{equation}\label{tausobest}
\|f\|_p + \sup_{j\ge 0} 2^{j\gamma}\|\tau_j(f)\|_p \sim \|f\|_{H_\gamma^p}.
\end{equation}
{\rm (c)} If $j\ge 2$, $d>0$, and $\nu\in \mathcal{R}_d$ then
\begin{equation}\label{taunuest}
\|\tau_j(f)\|_{\nu;1}\le (1+(2^jd)^q)^{1/p}|\!|\!|\nu|\!|\!|_{R,d}^{1/p}(|\nu|(\mathbb{X}))^{1/p'}\|\tau_j(f)\|_p.
\end{equation}
\end{lemma}

\begin{proof}\ 
Part (a) is an immediate consequence of \eqref{goodapprox}. Since $\Pi_{2^{j-2}}\subset \Pi_{2^{j-1}}$, \eqref{goodapprox} implies
$$
\|\tau_j(f)\|_p \le \|f-\sigma_{2^j}(f)\|_p+\|f-\sigma_{2^{j-1}}(f)\|_p\le c\mathsf{dist}(p;f,\Pi_{2^{j-2}}).
$$
This proves the first estimate in \eqref{taubasicest}. The second follows from \eqref{lpseries}. 
 The estimate \eqref{tausobest}  can be derived easily using \eqref{taubasicest} and Proposition~\ref{equivprop}. This completes the proof of part (b).

Next, we prove part (c). In this proof only, let 
$$
\tilde{G}(t)=h(t/2)-h(4t), \qquad g(t)=h(t)-h(2t).
$$
Then $g$ is supported on $[1/4,1]$, while
$$
\tilde{G}(t)=\left\{\begin{array}{ll}
0, &\mbox{ if $0\le t\le 1/8$,}\\
1, &\mbox{ if $1/4\le t\le 1$,}\\
0, &\mbox{ if $t\ge 2$.}
\end{array}\right.
$$
Therefore, it is easy to verify that $\tilde{G}(t)g(t)=g(t)$ for all $t$, $\tau_j(f)=\sigma_{2^j}(g;f)$, and hence, for all $f\in L^1$, $x\in\mathbb{X}$,
$$
\tau_j(f,x)=\int_\mathbb{X} \tau_j(f,y)\Phi_{2^j}(\tilde{G};x,y)d\mu^*(y).
$$
Using H\"older inequality followed by \eqref{phiintnorm} with $p'$ in place of $p$ and $\tilde{G}$ in place of $H$, we obtain that
\begin{eqnarray*}
\int_\mathbb{X} |\tau_j(f,x)|d|\nu|(x)&\le&\int_\mathbb{X} \int_\mathbb{X}|\Phi_{2^j}(\tilde{G};x,y)|d|\nu|(x)||\tau_j(f,y)|d\mu^*(y)\\
&\le &\left\|\int_\mathbb{X}|\Phi_{2^j}(\tilde{G};x,\circ)|d|\nu|(x)\right\|_{p'}\|\tau_j(f)\|_p\\
&\le& c\{(1+(d{2^j})^q)\}^{1/p}|\!|\!|\nu|\!|\!|_{R,d}^{1/p}(|\nu|(\mathbb{X}))^{1/p'}\|\tau_j(f)\|_p.
\end{eqnarray*}
This proves \eqref{taunuest}.
\end{proof}

We will use the following corollary of this lemma in our proofs.

\begin{corollary}\label{polyspaceequivcor}
If $n\ge 1$, $0<\gamma<r$, $P\in\Pi_n$, then 
\begin{equation}\label{polyspaceequiv1}
\sup_{\delta\in (0,1]}\frac{\omega_r(p;P,\delta)}{\delta^r}.\sim 
\|\Delta^r P\|_p\le cn^{r-\gamma}\|P\|_{H_\gamma^p}.
\end{equation}
Further,
\begin{equation}\label{polyspaceequiv}
\|P\|_{H^p_\gamma}\le c\{\|P\|_p+ \|\Delta^\gamma P\|_p \}\le cn^\gamma \|P\|_p.
\end{equation}
\end{corollary}
\begin{proof}\ 
In view of the fact that $\sigma_{2n}(P)=P$, we conclude from \eqref{jacksonest} used with $2n$ in place of $N$ that
$$
\frac{1}{(2n)^r}\|\Delta^r P\|_p=\|P-\sigma_{2n}(P)\|_p+\frac{1}{(2n)^r}\|\Delta^r\sigma_{2n}(P)\|_p \le c\omega_r(p;P,1/(2n)).
$$
This shows that 
$$
\|\Delta^r P\|_p \le c\sup_{\delta\in (0,1]}\frac{\omega_r(p;P,\delta)}{\delta^r}.
$$
The estimate in \eqref{polyspaceequiv1} in the other direction follows from the definition of $\omega_r(p;P,\delta)$.

Let $m$ be an integer, $2^m \le n<2^{m+1}$. Since the expansion for $P$ as given in \eqref{lpseries} is only a finite sum,  
we see that
$$
\|\Delta^r P\|_p=\left\|\sum_j \Delta^r \tau_j(P)\right\|_p =\left\|\sum_{j=0}^{m+2} \Delta^r\tau_j(P)\right\|_p\le \sum_{j=0}^{m+2}\|\Delta^r\tau_j(P)\|_p.
$$
Hence, using \eqref{pseudobern} and \eqref{tausobest}, we deduce that
$$
\|\Delta^r P\|_p \le c\sum_{j=0}^{m+2}2^{j(r-\gamma)}2^{j\gamma}\|\tau_j(P)\|_p
 \le c 2^{m(r-\gamma)}\|P\|_{H_\gamma^p}.
$$
This implies the last estimate in \eqref{polyspaceequiv1}.

If $N\ge n$ then $\mathsf{dist}(p, P,\Pi_N)=0$. If $N <n$, then the estimate \eqref{polyfavard}  yields
$$
\mathsf{dist}(p, P,\Pi_N)\le cN^{-\gamma}\|\Delta^\gamma P\|.
$$
Hence, Proposition~\ref{equivprop} shows that
$$
\|P\|_{H_\gamma^p} \sim \|P\|_p + \sup_{N\ge 1}N^\gamma \mathsf{dist}(p, P,\Pi_N) \le c\{\|P\|_p+\|\Delta^\gamma P\|\}.
$$
This proves the first estimate in \eqref{polyspaceequiv}; the second follows from \eqref{pseudobern}.
\end{proof}

Next, we recall   yet another preparatory lemma. The following lemma is proved in \cite[Lemma~5.4]{heatkernframe}. (In this lemma, the statement \eqref{prodapprox} is stated only for $p=\infty$, but the statement below follows since $\mu^*$ is a probability measure.)

\begin{lemma}\label{nikollemma}
Let $N\ge 1$, $P\in \Pi_N$, $0<p_1\le p_2\le \infty$. Then 
\begin{equation}\label{nikolskii}
\|P\|_{p_2}\le cN^{q(1/p_1-1/p_2)}\|P\|_{p_1}.
\end{equation}
 Further, let the product assumption hold,  $P_1, P_2\in\Pi_N$, $1\le p, p_1,p_2\le\infty$, and $R>0$ be arbitrary. Then there exists $Q\in \Pi_{A^*N}$ such that
\begin{equation}\label{prodapprox}
\|P_1P_2-Q\|_p \le c(R)N^{-R}\|P_1\|_{p_1}\|P_2\|_{p_2}.
\end{equation}
\end{lemma}

The following embedding theorem is a simple consequence of the results stated so far.
\begin{lemma}\label{embeddinglemma}
{\rm (a)} Let $1\le p_1<p_2\le\infty$, $\gamma>q(1/p_1-1/p_2)$, $f\in H_\gamma^{p_1}$. Then
\begin{equation}\label{nikembedineq}
 \|f\|_{H_{\gamma-q(1/p_1-1/p_2)}^{p_2}} \le c\|f\|_{H_{\gamma}^{p_1}}.
\end{equation}
{\rm (b)} Let $1\le p<\infty$, $\gamma>q/p$, and $f\in H^p_\gamma$. Then $f\in H^\infty_{\gamma-q/p}$ ( $f\in X^\infty$ in particular), and
\begin{equation}\label{contembedding}
\|f\|_{H_{\gamma-q/p}^\infty} \le c\|f\|_{H_\gamma^p}.
\end{equation}
\end{lemma}

\begin{proof}\ 
In this proof only, we will write $\alpha=q(1/p_1-1/p_2)$. Let $n\ge 0$,  $f\in H_\gamma^{p_1}$, and $r>\gamma$.  Without loss of generality, we may assume that $\|f\|_{H_\gamma^{p_1}} =1$. In view of \eqref{jacksonest},
$$
\frac{1}{2^{nr}}\|\Delta^r\sigma_{2^n}(f)\|_{p_1}\le c2^{-n\gamma}.
$$
 Since $
\Delta^r\sigma_{2^n}(f)\in \Pi_{2^n}$, Lemma~\ref{nikollemma} shows that
\begin{equation}\label{pf2eqn1}
\frac{1}{2^{nr}}\|\Delta^r\sigma_{2^n}(f)\|_{p_2} \le \frac{2^{n\alpha}}{2^{nr}}\|\Delta^r\sigma_{2^n}(f)\|_{p_1}\le c2^{-n(\gamma-\alpha)}.
\end{equation}
Further, since each $\tau_j(f)\in\Pi_{2^j}$, we deduce from \eqref{nikolskii}, \eqref{tausobest}, and the fact that $\gamma>\alpha$, that 
\begin{equation}\label{pf2eqn2}
\sum_{j=n+1}^\infty \|\tau_j(f)\|_{p_2} \le c\sum_{j=n+1}^\infty 2^{j\alpha}\|\tau_j(f)\|_{p_1}\le c\sum_{j=n+1}^\infty 2^{-j(\gamma-\alpha)} =c2^{-n(\gamma-\alpha)}.
\end{equation}
Consequently, the series
$$
\sigma_{2^n}(f)+\sum_{j=n+1}^\infty \tau_j(f)
$$
converges in $L^{p_2}$, necessarily to $f$. Therefore, $f\in X^{p_2}$. Further, \eqref{pf2eqn2} shows that
$$
\|f-\sigma_{2^n}(f)\|_{p_2} \le c2^{-n(\gamma-\alpha)}.
$$
Together with \eqref{pf2eqn1} we have thus shown that
$$
\|f-\sigma_{2^n}(f)\|_{p_2}+\frac{1}{2^{nr}}\|\Delta^r\sigma_{2^n}(f)\|_{p_2}\le c2^{-n(\gamma-\alpha)}.
$$
In view of \eqref{jacksonest}, this proves \eqref{nikembedineq}. 

Part (b) is special case of part (a).
\end{proof}

\section{Proofs of the Main Results}\label{pfsect}

We start with the proof of Theorem~\ref{maintheo}. This proof mimics that of \cite[Theorem~5]{besovquadpap}. However, while this theorem was proved in the case of the sphere for (exact) quadrature measures, the following theorem assumes only approximate quadratures and is, of course, valid for generic data-defined spaces.  \\

\noindent
\textit{Proof of Theorem~\ref{maintheo}.}\\

 Without loss of generality, we may assume in this proof that $\|f\|_{H_\gamma^p}=1$. 
Let $m\ge 0$ be an integer such that $2^m\le n<2^{m+1}$. Proposition~\ref{mzequivprop}(b) shows that 
$$
|\!|\!|\nu|\!|\!|_{R,2^{-m}}\sim |\!|\!|\nu|\!|\!|_{R,1/n}\sim |\!|\!|\nu|\!|\!|_{R,2^{-m-1}}\le c.
$$
Since $\gamma>q/p$, Lemma~\ref{embeddinglemma} shows that $f\in X^\infty$, so that Lemma~\ref{taulemma}(a) leads to
$$
f=\sigma_{2^m}(f)+\sum_{j=m+1}^\infty\tau_j(f),
$$
where the series converges uniformly. Hence, using \eqref{taunuest} with $d=2^{-m}$ and \eqref{tausobest}, we obtain
\begin{eqnarray}\label{pf4eqn1}
\left|\int_\mathbb{X} fd\nu -\int_\mathbb{X} \sigma_{2^m}(f)d\nu\right| &\le& \sum_{j=m+1}^\infty \left|\int_\mathbb{X} \tau_j(f)d\nu\right| \le \sum_{j=m+1}^\infty \|\tau_j(f)\|_{\nu;1}\nonumber\\
& \le& c|\!|\!|\nu|\!|\!|_{R,d}^{1/p}(|\nu|(\mathbb{X}))^{1/p'}\sum_{j=m+1}^\infty 2^{(j-m)q/p}\|\tau_j(f)\|_p\nonumber\\
& \le& c2^{-mq/p}|\!|\!|\nu|\!|\!|_{R,d}^{1/p}(|\nu|(\mathbb{X}))^{1/p'}\sum_{j=m+1}^\infty 2^{-j(\gamma-q/p)}\nonumber\\
&=&c2^{-m\gamma}|\!|\!|\nu|\!|\!|_{R,d}^{1/p}(|\nu|(\mathbb{X}))^{1/p'}.
\end{eqnarray}
In view of \eqref{approxquadpoly},  we obtain using Corollary~\ref{sigmahnormcor} that
$$
\left|\int_\mathbb{X} \sigma_{2^m}(f)d\mu^* - \int_\mathbb{X} \sigma_{2^m}(f)d\nu\right|\le \frac{A}{2^{m\gamma}}\|\sigma_{2^m}(f)\|_{H_\gamma^p} \le c\frac{A}{2^{m\gamma}}\|f\|_{H_\gamma^p} =c\frac{A}{2^{m\gamma}}.
$$
Using this observation and \eqref{pf4eqn1},
we deduce that
\begin{eqnarray*}
\left|\int_\mathbb{X} fd\mu^*-\int_\mathbb{X} fd\nu\right| &=& \left|\int_\mathbb{X} \sigma_{2^m}(f)d\mu^*-\int_\mathbb{X} fd\nu\right| \\
&\le& \left|\int_\mathbb{X} \sigma_{2^m}(f)d\mu^* - \int_\mathbb{X} \sigma_{2^m}(f)d\nu\right| + \left|\int_\mathbb{X} fd\nu -\int_\mathbb{X} \sigma_{2^m}(f)d\nu\right|\\
&\le& c\left(A+|\!|\!|\nu|\!|\!|_{R,d}^{1/p}(|\nu|(\mathbb{X}))^{1/p'}\right)2^{-m\gamma}.
\end{eqnarray*}
This proves Theorem~\ref{maintheo}. 
\qed

In order to prove Theorem~\ref{kernmintheo}, we first summarize in Proposition~\ref{networkkernprop} below some properties of the kernels $G$ introduced in Definition~\ref{eigkerndef}, including the existence of such a kernel. This proposition is proved in \cite[Proposition~5.2]{eignet}; we state it with $p'$ in \cite[Proposition~5.2]{eignet} replaced by $p$ per the requirement of our proof.
 Although the set up there is stated as that of a compact smooth manifold without boundary, the proofs are verbatim the same for data-defined spaces.

Let $b$ be a mask of type $\beta\in\mathbb{R}$. In the sequel, if $N>0$, we will write $b_N(t)=b(Nt)$. 

\begin{proposition}\label{networkkernprop}
Let $1\le p\le\infty$, $\beta>q/p$,  $G$ be a kernel of type $\beta$.\\
{\rm (a)} For every $y\in\mathbb{X}$, there exists $\psi_y:=G(\circ,y)\in X^{p'}$ such that $\langle \psi_y, \phi_k\rangle =b(\lambda_k)\phi_k(y)$, $k=0,1,\cdots$. We have
\begin{equation}\label{glpuniform}
\sup_{y\in\mathbb{X}}\|G(\circ,y)\|_{p'} \le c.
\end{equation}
{\rm (b)} Let 
 $n\ge 1$ be an integer, $\nu\in{\cal R}_{2^{-n}}$, and for $F\in L^1(\nu)\cap L^\infty(\nu)$, $m\ge n$,
$$
U_m(F,x):=\int_{y\in\mathbb{X}} \{G(x,y)-\Phi_{2^{m}}(hb_{2^{m}};x,y)\}F(y)d\nu(y).
$$
Then 
\begin{equation}\label{terrest}
\|U_m(F)\|_{p'} \le c2^{-m\beta}2^{q(m-n)/p}\|\nu\|_{R, 2^{-n}}\|F\|_{\nu;\mathbb{X},p'}.
\end{equation}
\end{proposition}
 
It is convenient to prove Proposition~\ref{kernminprop} before proving Theorem~\ref{kernmintheo}.

\vskip0.5cm
\noindent\textit{Proof of Proposition~\ref{kernminprop}.}\\

Let $P\in \Pi_{2^n}$, and we define
$$
\mathcal{D}_G(P)(x)=\sum_j \frac{\hat{P}(j)}{b(\lambda_j)}\phi_j(x), \qquad x\in\mathbb{X}.
$$
Then it is easy to verify that
\begin{equation}\label{pf6eqn3}
P(x)=\int_\mathbb{X} G(x,y)\mathcal{D}_G(P)(y)d\mu^*(y).
\end{equation}
Using Fubini's theorem and the condition that $M_p(\nu^\#)\le \tilde{A}2^{-n\beta}$, we deduce that
\begin{eqnarray}\label{pf6eqn4}
\lefteqn{\left|\int_\mathbb{X} P(x)d\nu^\#(x)-\int_\mathbb{X} P(x)d\mu^*(x)\right|}\nonumber\\
&=&\left|\int_\mathbb{X} \mathcal{D}_G(P)(y)\left\{\int_\mathbb{X} G(x,y)d\nu^\#(x)-\int_\mathbb{X} G(x,y)d\mu^*(x)\right\}d\mu^*(y)\right|\nonumber\\
&\le& \|\mathcal{D}_G(P)\|_p M_p(\nu^\#)\le \tilde{A}2^{-n\beta}\|\mathcal{D}_G(P)\|_p.
\end{eqnarray}
Let $0<\gamma<r<\beta$. We have proved in \cite[Lemma~5.4(b)]{eignet} that
$$
\|\mathcal{D}_G(P)\|_p \le c2^{n(\beta-r)}\|\Delta^r P\|_p.
$$
Hence, Corollary~\ref{polyspaceequivcor} implies that
$$
\|\mathcal{D}_G(P)\|_p \le c2^{n(\beta-\gamma)}\|P\|_{H^p_\gamma}.$$
Using  \eqref{pf6eqn4}, we now conclude that
$$
\left|\int_\mathbb{X} P(x)d\nu^\#(x)-\int_\mathbb{X} P(x)d\mu^*(x)\right|\le \tilde{A}2^{-n\beta}\|\mathcal{D}_G(P)\|_p  \le c\tilde{A}2^{-n\beta}2^{n(\beta-\gamma)}\|P\|_{H^p_\gamma}.
$$
Thus, $\nu^\#$ satisfies \eqref{approxquadpoly}.
\qed

\vskip0.5cm
\noindent\textit{Proof of Theorem~\ref{kernmintheo}.}\\

We note that 
$$
\int_\mathbb{X} \left\{G(x, y)-\Phi_{2^n}(hb_{2^n};x,y)\right\}d\mu^*(x)=0, \qquad y\in\mathbb{X}.
$$
Since $\nu^*\in K$, $\nu^*$ is a quadrature measure of order $2^n$, and $\Phi_{2^n}(hb_{2^n};x,\circ)\in\Pi_{2^n}$, we obtain that
\begin{eqnarray}\label{pf6eqn1}
M_p(\nu^\#) &\le& c\inf_{\nu\in K}M_p(\nu)\le cM_p(\nu^*)\nonumber\\
&=&c\left\|\int_\mathbb{X} \left\{G(x, \circ)-\Phi_{2^n}(hb_{2^n};x,\circ)\right\}d\nu^*(x)\right.\nonumber\\
&& \qquad\qquad \left.-\int_\mathbb{X} \left\{G(x, \circ)-\Phi_{2^n}(hb_{2^n};x,\circ)\right\}d\mu^*(x)\right\|_{p'}\nonumber\\
&=& c\left\|\int_\mathbb{X} \left\{G(x, \circ)-\Phi_{2^n}(hb_{2^n};x,\circ)\right\}d\nu^*(x)\right\|_{p'}.
\end{eqnarray}
We now use Proposition~\ref{networkkernprop}(b) with $F\equiv 1$,
$m=n$, to conclude that
\begin{equation}\label{pf6eqn2}
M_p(\nu^\#) \le c2^{-n\beta}.
\end{equation}
Thus, $\nu^\#$ satisfies the conditions in Proposition~\ref{kernminprop}, and hence, \eqref{approxquadpoly}. \qed 

The main idea in the proof of Theorem~\ref{covertheo} below is to show that the localization of the kernels $\Phi_N$ imply via Proposition~\ref{criticalprop} that the integral of $\Phi_N(x,\cdot)$ on $\mathbb{X}$ is concentrated on a ball of radius $\sim 1/N$ around $x$.\\

\vskip0.5cm
\noindent\textit{Proof of Theorem~\ref{covertheo}.}\\

Let $n\ge 1$,  $\nu=\nu_n\in\aleph$, and $x\in\mathbb{X}$. In this proof only, let  $\alpha\in (0,1)$ be fixed (to be chosen later), $N$ be defined by 
\begin{equation}\label{Nnrel}
N^{\gamma+q/p'}=(\alpha n)^\gamma; \quad\mbox{i.e.; } 
N= (\alpha n)^{1/\tilde{p}}.
\end{equation}
 We consider the polynomial $P=\Phi_{N}(x,\circ)$ and note that $P\in \Pi_N\subseteq \Pi_n$.
Since  $\nu$ satisfies \eqref{approxquadpoly}, 
\begin{equation}\label{pf5eqn2}
\left|1-\int_\mathbb{X} P(y)d\nu(y)\right|=\left|\int_\mathbb{X} P(y)d\mu^*(y)-\int_\mathbb{X} P(y)d\nu(y)\right| \le \frac{c}{n^{\gamma}}\|P\|_{H_\gamma^p}.
\end{equation}
In view of \eqref{polyspaceequiv} in Corollary~\ref{polyspaceequivcor}, and Proposition~\ref{equivprop},
we deduce using the definition \eqref{Nnrel} that
$$
\|P\|_{H_\gamma^p} \le cN^{\gamma}\|P\|_p \le cN^{\gamma+q/p'}\|P\|_1 \le cN^{\gamma+q/p'}=c\alpha^\gamma n^\gamma.
$$
Therefore, \eqref{pf5eqn2} leads to
$$
\left|1-\int_\mathbb{X} P(y)d\nu(y)\right| \le c\alpha^\gamma.
$$
We now choose $\alpha$ to be sufficiently small to ensure that
\begin{equation}\label{pf5eqn3}
\left|1-\int_\mathbb{X} \Phi_N(x,y)d\nu(y)\right|=\left|1-\int_\mathbb{X} P(y)d\nu(y)\right|\le 1/2, \qquad n\ge 1.
\end{equation}
Next, we use \eqref{phiintaway} with $h$ in place of $H$,  $d=1/n$, and $r=\lambda/N$ for sufficiently large $\lambda\ge 1$ to  be chosen later. Recalling that $N\le n$, this yields
\begin{eqnarray}\label{pf5eqn4}
\left|\int_{\Delta(x,\lambda/N)} P(y)d\nu(y)\right| &\le& \int_{\Delta(x,\lambda/N)} |\Phi_N(x,y)|d|\nu|(y)\nonumber\\
&\le& c(1+(N/n)^q)(\lambda )^{-S+q} \le 
c\lambda^{-S+q} ,
\end{eqnarray}
where we recall our convention that $|\!|\!|\nu|\!|\!|_{R,1/n}$ is assumed to be bounded independently of $n$. We now choose $\lambda$ to be large enough so that
\begin{equation}\label{pf5eqn5}
\int_{\Delta(x,\lambda/N)} |\Phi_N(x,y)|d|\nu|(y)\le 1/4.
\end{equation}
 Together with \eqref{pf5eqn3}, \eqref{pf5eqn4}, this leads to
\begin{equation}\label{pf5eqn6}
1/4\le \int_{\mathbb{B}
(x,\lambda/N)}\Phi_N(x,y)d\nu(y) \le 7/4, \qquad n\ge 1.
\end{equation}
Since $|\Phi_N(x,y)|\le cN^q$ (cf. \eqref{kernlocest}), we deduce that
\begin{eqnarray*}
1/4 &\le&  \int_{\mathbb{B}(x,\lambda/N)}\Phi_N(x,y)d\nu(y)\le  \int_{\mathbb{B}(x,\lambda/N)}|\Phi_N(x,y)|d|\nu|(y)\\
& \le& cN^q|\nu|(\mathbb{B}(x,\lambda/N)).
\end{eqnarray*}
This implies \eqref{nulowbd}.
\qed

Finally, the proof of Theorem~\ref{posimpliesreg} mimics that of \cite[Theorem~5.8]{modlpmz} (see also \cite[Theorem~5.5(a)]{modlpmz} to see the connection with regular measures). Unlike in that proof, we use only the approximate quadrature measures rather than exact quadrature measures.\\
 
\vskip0.5cm
\noindent\textit{Proof of Theorem~\ref{posimpliesreg}.}\\

Let $x\in\mathbb{X}$. With $A^*$ defined as in the product assumption and $\beta$ as in \eqref{phinlowbd}, let $\tilde{A}=\max(A^*,1/\beta)$.  In view of  \eqref{phinlowbd} and the fact that $\nu$ is a positive measure, we obtain  that
\begin{equation}\label{pf7eqn1}
n^{2q}\nu(\mathbb{B}(x,1/n)) \le c\int_{\mathbb{B}(x,1/n)}|\Phi_{n/\tilde{A}}(x,y)|^2d\nu(y) \le \int_\mathbb{X} |\Phi_{n/\tilde{A}}(x,y)|^2d\nu(y).
\end{equation}
Let $R\ge 1$. In view of \eqref{prodapprox} in Lemma~\ref{nikollemma}, there exists $Q\in\Pi_n$ such that 
\begin{equation}\label{pf7eqn2}
\|\Phi_{n/\tilde{A}}(x,\circ)^2-Q\|_\infty \le cn^{-R}\|\Phi_{n/\tilde{A}}(x,\circ)\|_1^2 \le cn^{-R}.
\end{equation}
Since $\nu$ satisfies \eqref{approxquadpoly}, we obtain first that 
$$
\left|\int_\mathbb{X} \phi_0d\mu^*-\int_\mathbb{X}\phi_0d\nu\right|\le c,
$$
so that $|\nu|(\mathbb{X})\le c$, and then conclude using \eqref{pf7eqn2} that
\begin{eqnarray}\label{pf7eqn3}
\lefteqn{\left|\int_\mathbb{X} |\Phi_{n/\tilde{A}}(x,y)|^2d\mu^*(y)-\int_\mathbb{X} |\Phi_{n/\tilde{A}}(x,y)|^2d\nu(y) \right|}\nonumber\\
&\le& \int_\mathbb{X} ||\Phi_{n/\tilde{A}}(x,y)|^2-Q(y)|d\mu^*(y)+\int_\mathbb{X} ||\Phi_{n/\tilde{A}}(x,y)|^2-Q(y)|d\nu(y) \nonumber\\
&&\quad+\left|\int_\mathbb{X} Qd\mu^*-\int_\mathbb{X} Qd\nu\right|\le \frac{c}{n^R}+c\frac{\|Q\|_{H_\gamma^p}}{n^\gamma}.
\end{eqnarray}
Since $Q\in\Pi_n$, Lemma~\ref{nikollemma}, Corollary~\ref{polyspaceequivcor}, and \eqref{pf7eqn2} lead to
\begin{eqnarray}\label{pf7eqn4}
\frac{\|Q\|_{H_\gamma^p}}{n^\gamma}&\le& c\|Q\|_p\le cn^{q/p'}\|Q\|_1\le cn^{q/p'}\left\{\|\Phi_{n/\tilde{A}}(x,\circ)^2-Q\|_\infty +\|\Phi_{n/\tilde{A}}(x,\circ)^2\|_1\right\}\nonumber\\
&\le& cn^{q/p'}\{n^{-R}+\|\Phi_{n/\tilde{A}}(x,\circ)^2\|_1\}.
\end{eqnarray}
In view of \eqref{philpnorm} in Proposition~\ref{criticalprop}, used with $\mu^*$ in place of $\nu$, $d=0$, $p=2$, we see that
$$
\|\Phi_{n/\tilde{A}}(x,\circ)^2\|_1=\|\Phi_{n/\tilde{A}}(x,\circ)\|_2^2\le cn^q.
$$
Therefore, \eqref{pf7eqn4} and \eqref{pf7eqn3} imply  that
$$
\int_\mathbb{X} |\Phi_{n/\tilde{A}}(x,y)|^2d\nu(y) \le cn^{q+q/p'}.
$$
Together with \eqref{pf7eqn1}, this leads to 
\eqref{regmeasbd}.
\qed

\begin{acknowledgement}
The research of this author is supported in part by ARO Grant W911NF-15-1-0385. We thank the referees for their careful reading and  helpful comments.
\end{acknowledgement}

%
%

\end{document}